\documentclass[a4paper,11pt]{article}
\usepackage{amsmath}
\usepackage{amssymb}
\usepackage{theorem}
\usepackage[usenames,dvipsnames]{pstricks}
\usepackage{pstricks-add}
\usepackage{pst-grad}
\usepackage{pst-plot} 
\usepackage{pst-eucl} 
\usepackage{euscript}
\usepackage{epic,eepic}
\usepackage{graphicx}
\usepackage{epsfig}
\newcommand{\email}[1]{\ttfamily #1}
\PassOptionsToPackage{normalem}{ulem}
\usepackage{ulem}

\renewcommand{\leq}{\ensuremath{\leqslant}}
\renewcommand{\geq}{\ensuremath{\geqslant}}

\newcommand{\Frac}[2]{\displaystyle{\frac{#1}{#2}}} 
\newcommand{\scal}[2]{{\left\langle{{#1}\mid{#2}}\right\rangle}}

\newcommand{\menge}[2]{\big\{{#1}~\big |~{#2}\big\}} 

\newcommand{\HH}{\ensuremath{{\mathcal H}}}

\newcommand{\emp}{\ensuremath{{\varnothing}}}

\newcommand{\Id}{\ensuremath{\operatorname{Id}}}

\newcommand{\RR}{\ensuremath{\mathbb{R}}}
\newcommand{\RP}{\ensuremath{\left[0,+\infty\right[}}

\newcommand{\RPP}{\ensuremath{\left]0,+\infty\right[}}

\newcommand{\NN}{\ensuremath{\mathbb N}}

\newcommand{\weakly}{\ensuremath{\:\rightharpoonup\:}}
\newcommand{\exi}{\ensuremath{\exists\,}}
\newcommand{\pinf}{\ensuremath{{+\infty}}}

\newcommand{\Fix}{\ensuremath{\operatorname{Fix}}}
\newcommand{\dom}{\ensuremath{\operatorname{dom}}}

\newcommand{\Argmin}{\ensuremath{\operatorname{Argmin}}}

\newcommand{\zeroun}{\ensuremath{\left]0,1\right[}}   
\newcommand{\zeroeta}{\ensuremath{\left]0,\eta\right[}}   
\newcommand{\rzeroun}{\ensuremath{\left]0,1\right]}}   
\newcommand{\xe}{x^\varepsilon}

\newtheorem{theorem}{Theorem}[section]
\newtheorem{lemma}[theorem]{Lemma}
\newtheorem{corollary}[theorem]{Corollary}
\newtheorem{proposition}[theorem]{Proposition}
\theoremstyle{plain}{\theorembodyfont{\rmfamily}%
\newtheorem{assumption}[theorem]{Assumption}}
\theoremstyle{plain}{\theorembodyfont{\rmfamily}%
}
\theoremstyle{plain}{\theorembodyfont{\rmfamily}%
\newtheorem{example}[theorem]{Example}}
\theoremstyle{plain}{\theorembodyfont{\rmfamily}%
\newtheorem{remark}[theorem]{Remark}}
\theoremstyle{plain}{\theorembodyfont{\rmfamily}%
\newtheorem{definition}[theorem]{Definition}}
\theoremstyle{plain}{\theorembodyfont{\rmfamily}%
}

\numberwithin{equation}{section}
\begin{document}
\title{\sffamily Asymptotic behavior of compositions of 
under-relaxed nonexpansive operators}
\author{J.-B. Baillon,$^1$ P. L. Combettes,$^{2}$ and 
R. Cominetti$^3$
\\[5mm]
\small
$\!^1$Universit\'e Paris 1 Panth\'eon-Sorbonne\\
\small SAMM -- EA 4543\\
\small 75013 Paris, France 
(\email{Jean-Bernard.Baillon@univ-paris1.fr})\\[4mm]
\small $\!^2$UPMC Universit\'e Paris 06\\
\small Laboratoire Jacques-Louis Lions -- UMR 7598\\
\small 75005 Paris, France (\email{plc@math.jussieu.fr})
\\[5mm]
\small
\small $\!^3$Universidad de Chile\\
\small Departamento de Ingenier\'{\i}a Industrial\\
\small Santiago, Chile (\email{rccc@dii.uchile.cl})
}
\date{~}
\maketitle

\vskip 8mm

\begin{abstract} 
\noindent
In general there exists no relationship between the fixed point sets 
of the composition and of the average of a family of nonexpansive 
operators in Hilbert spaces. In this paper, we establish an 
asymptotic principle connecting the cycles generated by 
under-relaxed compositions of nonexpansive operators to the fixed 
points of the average of these operators. 
In the special case when the operators are projectors onto 
closed convex sets, we prove a conjecture by De Pierro which  
has so far been established only for projections onto 
affine subspaces.
\end{abstract} 



\newpage

\section{Introduction}

Fixed points of compositions and averages of nonexpansive operators 
arise naturally in diverse settings; see for instance 
\cite{Banf11,Livre1,Byrn08,Cegi12} and the
references therein. In general there is no simple relationship
between the fixed point sets of such operators. In this paper we
investigate the connection of the fixed points of the average
operator with the limits of a family of under-relaxed compositions.
More precisely, we consider the framework described in the following
standing assumption.

\begin{assumption}
\em
\label{h:1}
$\HH$ is a real Hilbert space, $D$ is a nonempty, closed,
convex subset of $\HH$, $m\geq 2$ is an integer, 
$I=\{1,\ldots,m\}$, $(T_i)_{i\in I}$ is a family of nonexpansive 
operators from $D$ to $D$, and $(\Fix T_i)_{i\in I}$ is
their fixed point sets. Moreover, we set
\begin{equation}
\label{e:Reps}
\begin{cases}
T=\Frac{1}{m}\sum_{i\in I}T_i\\
R=T_m\circ\cdots\circ T_1\\
(\forall\varepsilon\in\zeroun)\;\;
R^{\varepsilon}=
\big(\Id+\varepsilon(T_m-\Id)\big)\circ\cdots\circ
\big(\Id+\varepsilon(T_{1}-\Id)\big).
\end{cases}
\end{equation}
\end{assumption}

When the operators $(T_i)_{i\in I}$
have common fixed points, 
$\Fix T=\bigcap_{i=1}^m\Fix T_i\neq\emp$ 
\cite[Proposition~4.32]{Livre1}. If, in addition, they 
are strictly nonexpansive in the sense that
\begin{equation}
\label{e:2012-11-20b}
(\forall i\in I)(\forall x\in D\smallsetminus\Fix T_i)
(\forall y\in\Fix T_i)\quad\|T_ix-y\|<\|x-y\|,
\end{equation}
it also holds that $\Fix R=\bigcap_{i\in I}\Fix T_i$
\cite[Corollary~4.36]{Livre1}, and therefore $\Fix R=\Fix T$. 
However, in the general case when $\bigcap_{i\in I}\Fix T_i=\emp$, 
the question has been long standing and remains open
even for convex projection operators, e.g., 
\cite[Section~8.3.2]{Byrn08} and \cite{Depi01}.
Even when $m=2$ and $T_1$ and $T_2$ are resolvents of maximally 
monotone operators, there does not seem to exist a simple 
relationship between $\Fix R$ and $\Fix T$ \cite{Wang11}, except 
for convex projection operators, in which case
$\Fix T=(1/2)(\Fix R+\Fix R')$, with $R'=T_1\circ T_2$
(see \cite{Baus93,Sign94} for related results, and \cite{Baus12} for 
the case of $m\geq 3$ resolvents). 

When $(T_i)_{i\in I}=(P_i)_{i\in I}$ are 
projection operators onto nonempty closed convex sets 
$(C_i)_{i\in I}$, $\Fix T$ is the set of minimizers of the 
average square-distance function \cite{Baus93,Sign94,Depi85}
\begin{equation}
\label{e:1994}
\Phi\colon\HH\to\RR\colon x\mapsto\frac{1}{2m} 
\sum_{i\in I}d_{C_i}^2(x),
\end{equation}
while $\Fix R$ is related to the set of Nash equilibria of a 
cyclic projection game. Indeed, the fixed point
equation $x=Rx$ can be restated as a system of equations in
$(x_1,\ldots,x_m)\in\HH^m$, namely
\begin{equation}
\label{e:2010-12-21cc}
\begin{cases}
x_1&=P_1x_m\\
x_2&=P_2x_1\\
&~\vdots\\
x_m&=P_mx_{m-1},
\end{cases}
\end{equation}
which characterize the Nash equilibria of a game in which each 
player $i\in I$ selects a strategy $x_i\in C_i$ to minimize the 
payoff $x\mapsto\|x-x_{i-1}\|$, with the convention 
$x_{0}=x_m$. 
It is worth noting that, for $m\geq 3$, these Nash equilibria
cannot be characterized as minimizers of any function 
$\Psi\colon\HH^m\to\RR$ over $C_1\times\cdots\times C_m$ 
\cite{Bail12}, which further reinforces the lack of hope for 
simple connections between $\Fix R$ and $\Fix T$. It was shown in 
\cite{Gubi67} that, if one of the sets is bounded, for every 
$y_0\in\HH$, the sequence 
$(y_{km+1},\ldots,y_{km+m})_{k\in\NN}$ generated 
by the periodic best-response dynamics
\begin{equation}
\label{e:pocs}
(\forall k\in\NN)\quad 
\begin{array}{l}
\left\lfloor
\begin{array}{ll}
y_{km+1}&=P_1y_{km}\\
y_{km+2}&=P_2y_{km+1}\\
&\;\vdots\\
y_{km+m}&=P_my_{km+m-1},
\end{array}
\right.\\[2mm]
\end{array}
\end{equation}
converges weakly to a solution $(x_1,\ldots,x_m)$ 
to \eqref{e:2010-12-21cc} (see Fig.~\ref{fig:1}).
\begin{figure}[h!tb]
\begin{center}
\scalebox{0.5}{
\begin{pspicture}(4,-4.85)(18.24,6.1)
\definecolor{color96b}{rgb}{0.80,1.0,1.0}
\rput{18.0}(-0.18953674,-2.8316424)%
{\psellipse[linewidth=0.06,dimen=outer,fillstyle=solid,%
fillcolor=color96b](8.844375,-2.0141652)(5.0,1.5)}
\psellipse[linewidth=0.06,dimen=outer,fillstyle=solid,%
fillcolor=color96b](11.444375,3.185835)(5.8,1.8)
\psellipse[linewidth=0.06,dimen=outer,fillstyle=solid,%
fillcolor=color96b](18.4,-2.18)(2.87,1.99)
\psline[linewidth=0.06cm,linestyle=dashed,linecolor=black,%
arrowsize=0.18cm 2.0,arrowlength=1.4,arrowinset=0.4]{->}%
(2.214375,0.20583488)(5.7243,2.926)
\psline[linewidth=0.06cm,linestyle=dashed,linecolor=black,%
arrowsize=0.18cm 2.0,arrowlength=1.4,arrowinset=0.4]{->}%
(5.7243,2.926)(7.304375,-0.994165)
\psline[linewidth=0.06cm,linestyle=dashed,linecolor=black,%
arrowsize=0.18cm 2.0,arrowlength=1.4,%
arrowinset=0.4]{->}(7.304375,-0.994165)(15.59,-1.9)
\psline[linewidth=0.06cm,linestyle=dashed,linecolor=black,%
arrowsize=0.18cm 2.0,arrowlength=1.4,%
arrowinset=0.4]{->}(15.59,-1.9)(14.70,1.74)
\psline[linewidth=0.06cm,linestyle=dashed,linecolor=black,%
arrowsize=0.18cm 2.0,arrowlength=1.4,%
arrowinset=0.4]{->}(14.70,1.74)(13.25,-0.11)
\psline[linewidth=0.06cm,arrowsize=0.18cm 2.0,arrowlength=1.4,%
arrowinset=0.4]{->}(15.1,1.83)(13.54,-0.29)
\psline[linewidth=0.06cm,arrowsize=0.18cm 2.0,arrowlength=1.4,%
arrowinset=0.4]{->}(13.5,-0.36)(15.75,-1.36)
\psline[linewidth=0.06cm,arrowsize=0.18cm 2.0,arrowlength=1.4,%
arrowinset=0.4]{->}(15.8,-1.36)(15.1,1.78)
\psdots[dotsize=0.18](2.24,0.22)
\rput(8.0,-2.8){\LARGE $C_2$}
\rput(11.2,3.6){\LARGE $C_1$}
\rput(18.8,-2.6){\LARGE $C_3$}
\rput(2.22,-0.20){\LARGE $y_0$}
\rput(7.5,-1.4){\LARGE $y_2$}
\rput(6.2,3.1){\LARGE $y_1$}
\rput(16.0,-2.2){\LARGE $y_3$}
\rput(14.3,2.04){\LARGE $y_4$}
\rput(12.9,0.3){\LARGE $y_5$}
\rput(13.1,-0.7){\LARGE ${x}_2$}
\rput(15.4,2.37){\LARGE ${x}_1$}
\rput(16.4,-1.33){\LARGE ${x}_3$}
\psdots[dotsize=0.18](15.1,1.83)
\psdots[dotsize=0.18](13.5,-0.36)
\psdots[dotsize=0.18](15.8,-1.36)
\end{pspicture} 
}
\caption{Cycles in the method of periodic projections.}
\end{center}
\label{fig:1}
\end{figure}
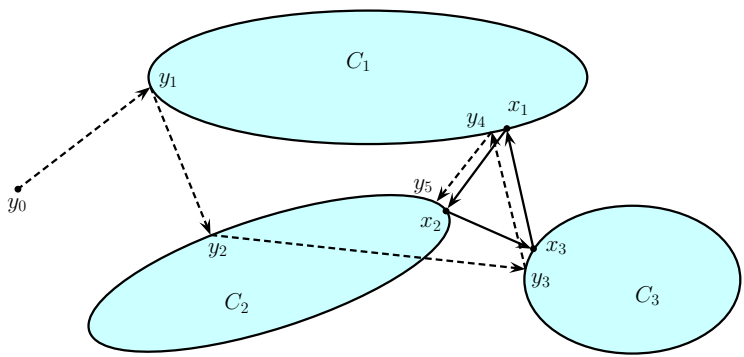
Working in a similar direction, and motivated by the work of 
\cite{Cens83} on under-relaxed projection methods for solving
inconsistent systems of affine inequalities, De Pierro considered
in \cite{Depi01} an under-relaxed version of \eqref{e:pocs},
namely 
\begin{equation}
\label{e:2012-04-09g}
(\forall k\in\NN)\quad 
\begin{array}{l}
\left\lfloor
\begin{array}{ll}
y^\varepsilon_{km+1}&=
\big(\Id+\varepsilon(P_1-\Id)\big)y^\varepsilon_{km}\\
y^\varepsilon_{km+2}&=
\big(\Id+\varepsilon(P_2-\Id)\big)
y^\varepsilon_{km+1}\\
&\;\vdots\\
y^\varepsilon_{km+m}&=
\big(\Id+\varepsilon(P_m-\Id)\big)
y^\varepsilon_{km+m-1}.
\end{array}
\right.\\[2mm]
\end{array}
\end{equation}
Under mild conditions the resulting sequence 
$(y^\varepsilon_{km+1},y^\varepsilon_{km+2},\ldots,%
y^\varepsilon_{km+m})_{k\in\NN}$ converges weakly 
to a limit cycle that satisfies the coupled equations
\begin{equation}
\label{e:pocs_limit}
(\forall\varepsilon\in\zeroun)\quad
\begin{cases}
x_1^\varepsilon&=
\big(\Id+\varepsilon(P_1-\Id)\big)x_{m}^\varepsilon\\
x_2^\varepsilon&=
\big(\Id+\varepsilon(P_2-\Id)\big)x_{1}^\varepsilon\\
&~\vdots\\       
x_m^\varepsilon&=
\big(\Id+\varepsilon(P_m-\Id)\big)x_{m-1}^\varepsilon.
\end{cases}
\end{equation}
In \cite[Conjecture~I]{Depi01}, De Pierro conjectured that as 
$\varepsilon\to 0$ these limit cycles
$(\xe_1,\ldots,\xe_m)_{\varepsilon\in\zeroun}$ shrink towards a
single point which is a minimizer of $\Phi$, i.e., a fixed
point of $T$. In contrast with \eqref{e:2010-12-21cc}, the
solutions of which do not satisfy any optimality criteria, this
conjecture suggests an asymptotic variational principle for the
cycles obtained as limits of the under-relaxed version of
\eqref{e:pocs}. An important contribution was made in 
\cite{Baus05}, where it was shown that De Pierro's conjecture
is true for families of closed affine subspaces which satisfy 
a certain regularity condition.

In this paper we investigate the asymptotic behavior of the
under-relaxed cycles 
\begin{equation}
\label{e:2011-12-06}
\begin{cases}
x_1^\varepsilon&=
\big(\Id+\varepsilon(T_1-\Id)\big)x_{m}^\varepsilon\\
x_2^\varepsilon&=
\big(\Id+\varepsilon(T_2-\Id)\big)x_{1}^\varepsilon\\
&~\vdots\\       
x_m^\varepsilon&=
\big(\Id+\varepsilon(T_m-\Id)\big)x_{m-1}^\varepsilon
\end{cases}
\end{equation}
as $\varepsilon\to 0$ in the general setting of Assumption~\ref{h:1}.
In Section~\ref{s2} we present a first general convergence result,
which establishes conditions under which the limits as 
$\varepsilon\to 0$ of the $m$ curves 
$(x_i^\varepsilon)_{\varepsilon\in\zeroun}$ ($i\in I$) exist and
all coincide with a fixed point of $T$. This result not only
gives conditions under which De Pierro's conjecture is true,
but also extends its scope from projection operators to
arbitrary nonexpansive operators.
In Section~\ref{s3} we revisit the problem from a 
constructive angle. Given an initial point $y_0\in D$
and $\varepsilon\in\zeroun$, it is known
\cite[Theorem~5.22]{Livre1} that the cycles in 
\eqref{e:2011-12-06} can be constructed iteratively
as the weak limit of the periodic process
\begin{equation}
\label{e:2012-04-09a}
(\forall k\in\NN)\quad 
\begin{array}{l}
\left\lfloor
\begin{array}{ll}
y^\varepsilon_{km+1}&=
\big(\Id+\varepsilon(T_1-\Id)\big)y^\varepsilon_{km}\\
y^\varepsilon_{km+2}&=
\big(\Id+\varepsilon(T_2-\Id)\big)
y^\varepsilon_{km+1}\\
&\;\vdots\\
y^\varepsilon_{km+m}&=
\big(\Id+\varepsilon(T_m-\Id)\big)
y^\varepsilon_{km+m-1}.
\end{array}
\right.\\[2mm]
\end{array}
\end{equation}
We analyze the connection between this iterative process
and the trajectories of the evolution equation
\begin{equation}
\label{e:2012-04-09b}
\begin{cases}
x'(t)+x(t)=Tx(t)\;\,\text{on}\;\RPP\\
x(0)=y_0,
\end{cases}
\end{equation}
and then establish extended versions of De Pierro's conjecture 
under various assumptions.  

\noindent
{\bfseries Notation.}
The scalar product of $\HH$ is denoted by $\scal{\cdot}{\cdot}$ 
and the associated norm by $\|\cdot\|$.
The symbols $\weakly$ and $\to$ denote, respectively, weak and 
strong convergence, and $\Id$ denotes the identity operator. 
The closed ball of center $x\in\HH$ and radius $\rho\in\RPP$
is denoted by $B(x;\rho)$. Given a nonempty closed convex 
subset  $C\subset \HH$, the distance function to $C$ and the
projection operator onto $C$ are respectively denoted by $d_C$ and
$P_C$.

\section{Convergence of general families of under-relaxed cycles}
\label{s2}

We investigate the asymptotic behavior of the cycles 
$(x_i^\varepsilon)_{i\in I}$ defined by \eqref{e:2011-12-06} 
when $\varepsilon\to 0$. Let us remark that such a cycle
$(x_i^\varepsilon)_{i\in I}$ is in bijection with the fixed points 
of the composition $R^\varepsilon$ of \eqref{e:Reps}.
Indeed, $z^\varepsilon=x_m^\varepsilon$ is a fixed point of 
$R^\varepsilon$; conversely, each 
$z^\varepsilon\in\Fix R^\varepsilon$ generates a cycle by 
setting, for every $i\in I$, 
$x_i^\varepsilon=(\Id+\varepsilon(T_i-\Id))x_{i-1}^\varepsilon$,
where $x_{0}^\varepsilon=z^\varepsilon$.
This motivates our second standing assumption.

\begin{assumption}
\label{h:2}
For every
$\varepsilon\in\zeroun$, $R^\varepsilon$ is given by 
\eqref{e:Reps} and 
\begin{equation}
\label{e:H}
(\exi\eta\in\left]0,1\right])(\exi\beta\in\RPP)
(\forall\varepsilon\in\left]0,\eta\right[)
(\exi z^{\varepsilon}\in\Fix R^{\varepsilon})\quad
\|z^{\varepsilon}\|\leq\beta.
\end{equation}
\end{assumption}

For later reference, we record the fact that under this assumption
the cycles in \eqref{e:2011-12-06} can be obtained as weak limits
of the iterative process \eqref{e:2012-04-09a}.

\begin{proposition}
\label{p:2013-04-22}
Suppose that Assumptions~\ref{h:1} and \ref{h:2} are satisfied.
Let $y_0\in D$ and $\varepsilon\in\zeroeta$. Then the sequence 
$(y^\varepsilon_{km+1},\ldots,y^\varepsilon_{km+m})_{k\in\NN}$ 
produced by \eqref{e:2012-04-09a} converges weakly to an $m$-tuple
$(x^\varepsilon_{1},\ldots,x^\varepsilon_{m})$ which
satisfies \eqref{e:2011-12-06}.
\end{proposition}
\begin{proof}
This follows from \cite[Theorem~5.22]{Livre1}.
\end{proof}

The following result provides sufficient conditions for 
Assumption~\ref{h:2} to hold.
\begin{proposition} 
\label{p:zenzile}
Suppose that Assumption~\ref{h:1} holds, together with one 
of the following.
\begin{enumerate}
\item 
\label{p:zenzilei}
For some $j\in I$, $T_j$ has bounded range.
\item
\label{p:zenzileii}
$D$ is bounded.
\end{enumerate}
Then Assumption~\ref{h:2} is satisfied.
\end{proposition}
\begin{proof}
It is clear that \ref{p:zenzileii} is a special case of 
\ref{p:zenzilei}. Suppose that \ref{p:zenzilei} holds.
Fix $\varepsilon\in\rzeroun$ and $y\in D$, and take 
$\rho\in\left[\max_{i\in I\smallsetminus\{j\}}
\|T_iy-y\|,\pinf\right[$ 
such that $T_j(D)\subset B(y;\rho)$.
Furthermore, let $x\in D$, set $x_0=x$, and define recursively
$x_i=(1-\varepsilon)x_{i-1}+\varepsilon T_ix_{i-1}$, 
so that $x_m=R^\varepsilon x$. Then
\begin{align}
\label{e:kj1}
(\forall i\in I\smallsetminus\{j\})\quad
\|x_i-y\|
&= \|(1-\varepsilon)(x_{i-1}-y)+\varepsilon(
T_ix_{i-1}-y)\|\nonumber\\
&\leq (1-\varepsilon)\|x_{i-1}-y\|+
\varepsilon\|T_ix_{i-1}-T_iy\|+\varepsilon\|T_iy-y\|\nonumber\\
&\leq \|x_{i-1}-y\|+\varepsilon\rho
\end{align}
and
\begin{align}
\label{e:kj2}
\|x_{j}-y\|
&\leq (1-\varepsilon)\|x_{j-1}-y\|+
\varepsilon\|T_{j}x_{j-1}-y\|\nonumber\\
&\leq (1-\varepsilon)\|x_{j-1}-y\|+\varepsilon\rho.
\end{align}
By applying inductively \eqref{e:kj1} and \eqref{e:kj2} to
majorize $\|x_m-y\|$, we obtain
\begin{equation}
\label{e:bluemask1982}
\|R^\varepsilon x-y\|=\|x_m-y\|\leq (1-\varepsilon)\|x-y\|
+\varepsilon m\rho.
\end{equation}
This implies that $R^{\varepsilon}$ maps $D\cap B(y;m\rho)$ to 
itself. Hence, the Browder--G\"ohde--Kirk theorem (see 
\cite[Theorem~4.19]{Livre1}) asserts that $R^\varepsilon$
has a fixed point in $B(y;m\rho)$. 
Moreover, if $x$ is a fixed point of $R^{\varepsilon}$,
\eqref{e:bluemask1982} gives $\|x-y\|\leq m\rho$, which shows 
that \eqref{e:H} holds with $\eta=1$ and $\beta=\|y\|+m\rho$.
\end{proof}

To illustrate Assumption~\ref{h:2}, it is instructive to consider 
the following examples.

\begin{example}
\label{ex:2013-04-06}
The following variant of the example discussed in 
\cite[Section~3]{Depi01} shows that \eqref{e:H} is a non trivial 
assumption: $\HH$ is the Euclidean plane, $m=3$,
$\alpha\in\RR$, $\beta\in\RR$, $\gamma\in\RPP$, 
$\varepsilon\in\zeroun$, and $(T_i)_{1\leq i\leq 3}$ are, 
respectively, the projection operators onto the sets
\begin{equation}
\label{e:2013-04-05}
C_1=\RR\times\{\alpha\},\quad
C_2=\RR\times\{\beta\},\quad\text{and}\quad
C_3=\menge{(\xi_1,\xi_2)\in\RPP^2}{\xi_1\xi_2\geq\gamma}.
\end{equation}
Then we have
\begin{equation}
\begin{cases}
\Fix T=\menge{(\xi_1,\xi_2)\in C_3}{\xi_2=(\alpha+\beta)/2}\\
\Fix R =\menge{(\xi_1,\xi_2)\in C_3}{\xi_2=\beta}\\
\Fix R^\varepsilon=\menge{(\xi_1,\xi_2)\in C_3}
{\xi_2=\big((1-\varepsilon)\alpha+\beta\big)/(2-\varepsilon)}.
\end{cases}
\end{equation}
Thus, depending on the values of $\alpha$ and $\beta$, we can have
$\Fix T=\Fix R\neq\emp$, $\Fix T=\Fix R=\emp$, 
$\Fix T\neq\Fix R=\emp$, $\Fix R\neq\Fix T=\emp$, or
$\emp\neq\Fix R\neq\Fix T\neq\emp$.  
Now set $\eta=1+\beta/\alpha$. Then, under the assumption that
$\alpha+\beta<0<\beta$, we have $\eta\in\zeroun$ and  
$\Fix R^\varepsilon=\emp$ if $\varepsilon\leq\eta$, while
$\Fix R^\varepsilon\neq\emp$ if $\varepsilon>\eta$.
On the other hand, under the assumption that
$\beta<0<\alpha+\beta$, $\eta\in\zeroun$ and  
$\Fix R^\varepsilon\neq\emp$ if $\varepsilon<\eta$, while
$\Fix R^\varepsilon=\emp$ if $\varepsilon\geq\eta$. 
Moreover, setting
\begin{equation}
\label{e:2013-04-17}
(\forall\varepsilon\in\zeroeta)\quad
\begin{cases}
y^\varepsilon=\bigg(\Frac{2\gamma}
{(1-\varepsilon)\alpha+\beta}+\frac{1}{\varepsilon}\,,
\Frac{(1-\varepsilon)\alpha+\beta}{2-\varepsilon}\bigg)
\in\Fix R^\varepsilon\\[4mm]
z^\varepsilon=\bigg(\Frac{(2-\varepsilon)\gamma}
{(1-\varepsilon)\alpha+\beta}\,,
\Frac{(1-\varepsilon)\alpha+\beta}{2-\varepsilon}\bigg)
\in\Fix R^\varepsilon.
\end{cases}
\end{equation}
we see that
$(y^\varepsilon)_{\varepsilon\in\zeroeta}$ is an unbounded curve, 
while $(z^\varepsilon)_{\varepsilon\in\zeroeta}$ is bounded.
\end{example}

\begin{example}
In Example~\ref{ex:2013-04-06} the sets $(\Fix T_i)_{1\leq i\leq 3}$
are nonempty, and one may ask whether this plays a role in the 
nonemptiness of $\Fix R$, $\Fix T$, or $\Fix R^\varepsilon$. To 
see that such is not the case, define $T_3$ as in 
Example~\ref{ex:2013-04-06}, and consider the modified operators 
$T_1\colon(\xi_1,\xi_2)\mapsto(\xi_1+\mu,\alpha)$ and $T_2\colon
(\xi_1,\xi_2)\mapsto(\xi_1-\mu,\beta)$, where $\mu>0$. 
Although now the nonexpansive operators $T_1$ and $T_2$ have no 
fixed points, the operators $T$, $R$, and $R^\varepsilon$ remain 
unchanged.
\end{example}

\begin{example}
\label{ex:2013-04-08b}
By considering products of sets of the form \eqref{e:2013-04-05} 
one can build an example in which $\Fix T$ is nonempty but the
sets $(\Fix R^\varepsilon)_{\varepsilon\in\zeroun}$ are empty.
More precisely, let $\HH=\ell^2(\NN)$, and let $(\alpha_n)_{n\in\NN}$, 
$(\beta_n)_{n\in\NN}$, and  
$(\gamma_n)_{n\in\NN}$ be sequences in $\ell^2(\NN)$ such that 
$({\gamma_n}/(\alpha_n+\beta_n))_{n\in\NN}\in\ell^2(\NN)$
and $(\forall n\in\NN)$ $\beta_n<0<\alpha_n+\beta_n$ 
and $\gamma_n>0$. Set
\begin{equation}
\label{e:2013-04-08}
\begin{cases}
C_1=\menge{(\xi_n)_{n\in\NN}\in\ell^2(\NN)}{(\forall n\in\NN)\;\;
\xi_{2n}=\alpha_{n}}\\
C_2=\menge{(\xi_n)_{n\in\NN}\in\ell^2(\NN)}{(\forall n\in\NN)\;\;
\xi_{2n}=\beta_{n}}\\
C_3=\menge{(\xi_n)_{n\in\NN}\in\ell^2(\NN)}
{(\forall n\in\NN)\;\;\xi_n>0\;\;\text{and}\;\;
\xi_{2n-1}\xi_{2n}\geq\gamma_{n}}.
\end{cases}
\end{equation}
Then $\Fix T\neq\emp$ but, for $\varepsilon\in\zeroun$,
we have $\Fix R^{\varepsilon}\neq\emp$  
if and only if $(\forall n\in\NN)$ 
$\varepsilon<1+\beta_{2n+1}/\alpha_{2n+1}$.
In particular if we take, for every $n\in\NN\smallsetminus\{0\}$,
$\alpha_n=(n+1)/n^2$, $\beta_n=-1/n$, and
$\gamma_n=1/n^3$, then $\Fix R^\varepsilon=\emp$ for every 
$\varepsilon\in\zeroun$.
\end{example}

\begin{example}{\rm(\cite[Example~4.1]{Baus05})}
\label{ex:2013-04-08c}
Let $m=2$, and let $T_1$ and $T_2$ be the projection operators onto
closed affine subspaces $C_1\subset\HH$ and $C_2\subset\HH$, 
respectively. If $\HH$ is finite-dimensional, the sets $\Fix R$, 
$(\Fix R^\varepsilon)_{\varepsilon\in\zeroun}$, and $\Fix T$ 
are nonempty; if $\HH$ is infinite-dimensional, there exist $C_1$ 
and $C_2$ such that these sets are all empty. However, if 
the vector subspace $(C_1-C_1)+(C_2-C_2)$ is closed, then 
$\Fix T\neq\emp$ and $(\forall\varepsilon\in\zeroun)$
$\Fix R^\varepsilon\neq\emp$.
\end{example}

The next result establishes conditions for the convergence of the 
cycles of \eqref{e:2011-12-06} when the relaxation parameter
$\varepsilon$ vanishes. 

\begin{theorem}
\label{t:1}
Suppose that Assumptions~\ref{h:1} and \ref{h:2} are satisfied. 
Then $\Fix T\neq\emp$. Now let 
$(x_m^{\varepsilon})_{\varepsilon\in\zeroeta}=
(z^{\varepsilon})_{\varepsilon\in\zeroeta}$ be the bounded
curve provided by \eqref{e:H} and denote by
$(x_1^\varepsilon,\ldots,x_m^\varepsilon)_{\varepsilon\in\zeroeta}$ 
the associated family of cycles arising from \eqref{e:2011-12-06}.
Then 
$(x_1^\varepsilon,\ldots,x_m^\varepsilon)_{\varepsilon\in\zeroeta}$ 
is bounded and each of its weak sequential cluster points
is of the form $(x,\ldots,x)$, where $x\in\Fix T$. Moreover,
\begin{equation}
\label{e:2013-03-24}
(\forall i\in I)\quad 
\lim_{\varepsilon\to 0}\|x^\varepsilon_{i}-x^\varepsilon_{i-1}\|=0,
\quad\text{where}\quad (\forall\varepsilon\in\zeroeta)
\quad x^\varepsilon_{0}=x^\varepsilon_{m}.
\end{equation}
In addition, suppose that one of the following holds.
\begin{enumerate}
\item
\label{t:1i} 
$(\forall x\in\Fix T)(\forall y\in\Fix T)$ 
$\scal{x_m^\varepsilon}{x-y}$ converges as $\varepsilon\to 0$.
\item
\label{t:1ii} 
$(\forall x\in\Fix T)$ $\|x_m^\varepsilon-x\|$ converges as 
$\varepsilon\to 0$.
\item
\label{t:1iv} 
$\Fix T$ is a singleton.
\end{enumerate}
Then there exists $\overline{x}\in\Fix T$ such that, 
for every $i\in I$, $x_i^\varepsilon\weakly\overline{x}$ as 
$\varepsilon\to 0$. Finally,
suppose that $\Id-T$ is demiregular on $\Fix T$, i.e., 
\begin{equation}
\label{e:condi2}
\big(\forall (y_k)_{k\in\NN}\in D^{\NN}\big)
\big(\forall y\in\Fix T\big)\quad
\begin{cases}
y_k\weakly y\\
y_k-Ty_k\to 0
\end{cases}
\quad\Rightarrow\quad y_k\to y.
\end{equation}
Then, for every 
$i\in I$, $x_i^\varepsilon\to\overline{x}$ as $\varepsilon\to 0$.
\end{theorem}
\begin{proof}
Fix $z\in D$. By nonexpansiveness of the operators 
$(T_i)_{i\in I}$, we have
\begin{align}
\label{e:2012-04-08x}
(\forall i\in I)\quad
\|T_ix^\varepsilon_{i-1}-x^\varepsilon_{i-1}\|
&\leq\|T_ix^\varepsilon_{i-1}-T_iz\|+\|T_iz-z\|
+\|z-x^\varepsilon_{i-1}\|\nonumber\\
&\leq2\|x^\varepsilon_{i-1}-z\|+\|T_iz-z\|.
\end{align}
In particular, for $i=1$, it follows from the boundedness of 
$(x_m^{\varepsilon})_{\varepsilon\in\zeroeta}$ that
$(T_1x_m^{\varepsilon}-x_m^{\varepsilon})_{\varepsilon\in\zeroeta}$ 
is bounded. In turn, we deduce from \eqref{e:2011-12-06} that 
$(x_1^{\varepsilon})_{\varepsilon\in\zeroeta}$ is bounded.
Continuing this process, we obtain the boundedness of 
$(x_1^\varepsilon,\ldots,x_m^\varepsilon)_{\varepsilon\in\zeroeta}$ 
and the fact that
\begin{equation}
\label{e:2013-04-08c}
(\forall i\in I)\quad (T_ix_{i-1}^{\varepsilon}-
x_{i-1}^{\varepsilon})_{\varepsilon\in\zeroeta}
\quad\text{is bounded.}
\end{equation}
On the other hand, adding all the equalities in 
\eqref{e:2011-12-06}, we get
\begin{equation}
\label{e:2011-12-08b}
(\forall\varepsilon\in\zeroeta)\quad
\sum_{i\in I}T_ix_{i-1}^\varepsilon=
\sum_{i\in I}x_i^\varepsilon,
\end{equation}
from which it follows that
\begin{align}
\label{e:2011-12-09}
(\forall\varepsilon\in\zeroeta)\quad
Tx_m^\varepsilon-x_m^\varepsilon
&=\frac1m\sum_{i=1}^m T_ix_m^\varepsilon-x_m^\varepsilon
\nonumber\\
&=\frac1m\sum_{i=1}^m T_ix_{i-1}^\varepsilon+
\frac1m\sum_{i=2}^{m}\big(T_ix_m^\varepsilon-
T_ix_{i-1}^\varepsilon\big)-x_m^\varepsilon\nonumber\\
&=\frac1m\sum_{i=1}^m x_i^\varepsilon+
\frac1m\sum_{i=2}^{m}\big(T_ix_m^\varepsilon-
T_ix_{i-1}^\varepsilon\big)-x_m^\varepsilon\nonumber\\
&=\frac1m\sum_{i=1}^{m-1}(x_i^\varepsilon-x_m^\varepsilon)+
\frac1m\sum_{i=1}^{m-1}\big(T_{i+1}x_m^\varepsilon-
T_{i+1}x_i^\varepsilon\big).
\end{align}
Hence, using the nonexpansiveness of the operators 
$(T_i)_{i\in I}$, we obtain
\begin{align}
\label{e:2012-01-03}
(\forall\varepsilon\in\zeroeta)\quad
\|Tx_m^\varepsilon-x_m^\varepsilon\|
&\leq\frac2m\sum_{i=1}^{m-1}\big\|x_m^\varepsilon-
x_i^\varepsilon\big\|.
\end{align}
Consequently, since \eqref{e:2011-12-06} and \eqref{e:2013-04-08c} 
also imply that
\begin{equation}
\label{e:2011-12-07}
(\forall i\in I)\quad
\|x_i^\varepsilon-x_{i-1}^\varepsilon\|
=\varepsilon\|T_ix_{i-1}^\varepsilon-x_{i-1}^\varepsilon\|
\to 0\quad\text{as}\quad\varepsilon\to 0,
\end{equation}
thus proving \eqref{e:2013-03-24}, the triangle inequality gives 
$\|x_m^\varepsilon-x_i^\varepsilon\|\to 0$,
which, combined with \eqref{e:2012-01-03}, yields 
\begin{equation}
\label{e:2013}
Tx_m^\varepsilon-x_m^\varepsilon\to 0. 
\end{equation}
Hence, we can 
invoke the demiclosed principle \cite[Corollary~4.18]{Livre1} 
to deduce that every weak sequential cluster point of the 
bounded curve
$(x_m^\varepsilon)_{\varepsilon\in\zeroeta}$ belongs to $\Fix T$,
which is therefore nonempty. In view of 
\eqref{e:2011-12-07}, we therefore deduce that 
every weak sequential cluster point of 
$(x_1^\varepsilon,\ldots,x_m^\varepsilon)_{\varepsilon\in\zeroeta}$ 
is of the form $(x,\ldots,x)$, where $x\in\Fix T$.
It remains to show that under any of the conditions \ref{t:1i},
\ref{t:1ii}, or \ref{t:1iv}, the curve
$(x_m^\varepsilon)_{\varepsilon\in\zeroun}$ is weakly convergent.
Clearly \ref{t:1iv} implies \ref{t:1i}, and the same holds for
\ref{t:1ii} since
\begin{equation}
\label{e:2012-11-28}
(\forall(x,y)\in\HH^2)(\forall\varepsilon\in\zeroeta)\quad
\scal{x_m^\varepsilon}{x-y}=\frac12\big(\|x_m^\varepsilon-y\|^2
-\|x_m^\varepsilon-x\|^2+\|x\|^2-\|y\|^2\big).
\end{equation}
Thus, it suffices to show that under \ref{t:1i} the curve
$(x_m^\varepsilon)_{\varepsilon\in\zeroeta}$ has a unique 
weak sequential cluster point.
Let $x$ and $y$ be two weak sequential cluster points and
choose sequences $(\varepsilon_n)_{n\in\NN}$ and
$(\varepsilon'_n)_{n\in\NN}$ in $\zeroeta$ converging to $0$
such that $x_m^{\varepsilon_n}\weakly x$ and 
$x_m^{\varepsilon'_n}\weakly y$ as $n\to\pinf$.
As shown above, we have $x$ and $y$ lie in $\Fix T$ and, therefore,
it follows from \ref{t:1i} that $\scal{x}{x-y}=\lim_{n\to\pinf}
\scal{x_m^{\varepsilon_n}}{x-y}=\lim_{n\to\pinf}
\scal{x_m^{\varepsilon'_n}}{x-y}=\scal{y}{x-y}$. 
This yields $\|x-y\|^2=0$ proving our claim. 

Finally, let us establish the strong convergence assertion. 
To this end, let $(\varepsilon_n)_{n\in\NN}$ be a sequence
in $\zeroeta$ converging to $0$. Then, as just proved,
$x_m^{\varepsilon_n}\weakly\overline{x}\in\Fix T$ as $n\to\pinf$. 
On the other hand, \eqref{e:2013} yields 
$x_m^{\varepsilon_n}-Tx_m^{\varepsilon_n}\to 0$ as $n\to\pinf$. 
Hence, we derive from \eqref{e:condi2} that 
$x_m^{\varepsilon_n}\to\overline{x}$ as $n\to\pinf$.
This shows that $x_m^{\varepsilon}\to\overline{x}$ as 
$\varepsilon\to 0$. In view of \eqref{e:2011-12-07}, 
the proof is complete.
\end{proof}

\begin{remark}
\label{r:2013-04-23}
The demiregularity condition \eqref{e:condi2} is a specialization
of a notion introduced in \cite[Definition~2.3]{Sico10} for
set-valued operators (see also \cite[Definition~27.1]{Zeid90}). 
It follows from \cite[Proposition~2.4]{Sico10} that 
\eqref{e:condi2} is satisfied in each of the following cases. 
\begin{enumerate}
\item
\label{p:2009-09-20i}
$\Id-T$ is uniformly monotone at every $y\in\Fix T$.
\item
\label{r:2013-04-23ii}
$\Id-T$ is strongly monotone at every $y\in\Fix T$.
\item
\label{p:2009-09-20i+}
$T=\Id-\nabla f$, where $f\in\Gamma_0(\HH)$ is uniformly convex at
every $y\in\Fix T$.
\item
\label{p:2009-09-20iv}
$D$ is boundedly compact: its intersection 
with every closed ball is compact.
\item
\label{p:2009-09-20vi}
$D=\HH$ and $\Id-T$ is invertible.
\item
\label{p:2009-09-20vii}
$T$ is demicompact \cite{Petr66}: for every bounded sequence 
$(y_n)_{n\in\NN}$ in $D$ such that $(y_n-Ty_n)_{n\in\NN}$ converges
strongly, $(y_n)_{n\in\NN}$ admits a strongly convergent 
subsequence.
\end{enumerate}
\end{remark}

In the special case when $(T_i)_{i\in I}$ is a family of projection
operators onto closed convex sets, Theorem~\ref{t:1} asserts that
De Pierro's conjecture is true under any of conditions
\ref{t:1i}--\ref{t:1iv}. In particular, we obtain weak
convergence of each point in the cycle to the point in $\Fix T$ 
if this set is a singleton, which can be considered as a 
generic situation in many practical instances when 
$\bigcap_{i\in I}\Fix T_i\neq\emp$. 
The following example illustrates a degenerate case 
in which weak convergence of the cycles can fail.

\begin{example}
\label{ex:2012}
Suppose that in Theorem~\ref{t:1} we have
$\bigcap_{i\in I}\Fix T_i\neq\emp$. Then  it follows from the 
results of \cite[Section~4.5]{Livre1} that
\begin{equation}
\label{e:2012-01-30}
(\forall\varepsilon\in\zeroun)\quad
\Fix R^\varepsilon=\bigcap_{i\in I}\Fix\big((1-\varepsilon)\Id
+\varepsilon T_i\big)=
\bigcap_{i\in I}\Fix T_{i}=\Fix T.
\end{equation}
Now suppose $y$ and $z$ are two distinct points in $\Fix T$ and
set
\begin{equation}
\label{e:2012-01-30b}
(\forall\varepsilon\in\zeroun)\quad
x_{m}^{\varepsilon}=
\begin{cases}
y,&\text{if}\;\;\lfloor 1/\varepsilon \rfloor\;\text{is even};\\
z,&\text{if}\;\;\lfloor 1/\varepsilon \rfloor\;\text{is odd}.
\end{cases}
\end{equation}
Then $(x_{m}^{\varepsilon})_{\varepsilon\in\zeroun}$ has two 
distinct weak cluster points and therefore it does not converge 
weakly, although Assumptions~\ref{h:1} and \ref{h:2} are 
trivially satisfied. 
\end{example}

\section{Convergence of limit cycles of under-relaxed iterations}
\label{s3}

As illustrated in Example~\ref{ex:2012}, in general one cannot 
expect every solution cycle
$(x_1^\varepsilon,\ldots,x_m^\varepsilon)_{\varepsilon\in\zeroeta}$
in \eqref{e:2011-12-06} to converge as there are cases 
that oscillate. Theorem~\ref{t:1} provided conditions that rule out
multiple clustering and ensure the weak convergence of the 
cycles as $\varepsilon\to 0$. 
An alternative approach, inspired from \cite{Depi01}, 
is to focus on solutions of \eqref{e:2011-12-06} that arise 
as limit cycles of the under-relaxed periodic iteration
\eqref{e:2012-04-09a} started from the same initial point 
$y_0\in D$ for every $\varepsilon\in\zeroeta$. 
This arbitrary but fixed initial point is 
intended to act as an anchor that avoids multiple cluster points
of the resulting family of limit cycles
$(x_1^\varepsilon,\ldots,x_m^\varepsilon)_{\varepsilon\in\zeroeta}$.

As mentioned in the Introduction, for convex projection operators 
De Pierro conjectured that, as $\varepsilon\to 0$, the limit 
cycles shrink to a least-squares solution, 
namely $(x_1^\varepsilon,\ldots,x_m^\varepsilon)\weakly 
(\overline{x},\ldots,\overline{x})$, where $\overline{x}$ is a 
minimizer of the function $\Phi$ of \eqref{e:1994}. 
In \cite[Theorem~6.4]{Baus05} the conjecture was proved for 
closed affine subspaces satisfying a regularity conditions,
in which case the limit $\overline{x}$ exists in the strong topology 
and is in fact the point in $S=\Argmin\Phi=\Fix T$ closest to the 
initial point $y_0$, namely $\overline{x}=P_Sy_0$. 
However, for general convex sets the conjecture remains open.

We revisit this question in the general framework delineated
by Assumptions~\ref{h:1} and \ref{h:2} with a different strategy
than that adopted in Section~\ref{s2}.
Our approach consists in showing that, for
$\varepsilon$ small, the iterates \eqref{e:2012-04-09a} follow 
closely the orbit of the semigroup generated by $A=\Id-T$, 
i.e., the semigroup associated with the autonomous Cauchy problem
\begin{equation}
\label{gradflow_T}
\begin{cases}
x'(t)=-Ax(t)\;\,\text{on}\;\RPP\\
x(0)=y_0.
\end{cases}
\end{equation}
This allows us to relate the limit cycles 
$(x_1^\varepsilon,\ldots,x_m^\varepsilon)_{\varepsilon\in\zeroeta}$ 
to the limit of $x(t)$ when $t\to\pinf$. Note that, since 
$y_0\in D=\dom A$ and $A$ is Lipschitz, \eqref{gradflow_T} has 
a unique solution $x\in{\EuScript C}^1(\RPP;D)$;
see, e.g., \cite[Theorem~I.1.4]{Brez73}. In addition, if
there exists $x_\infty\in\HH$ such that $x(t)\weakly x_\infty$ as
$t\to\pinf$, then $\overline{x}\in\Fix T$.
In the case of convex projections, \eqref{gradflow_T} 
reduces to the gradient flow
\begin{equation}
\label{gradflow}
\begin{cases}
x'(t)=-\nabla\Phi(x(t))\;\,\text{on}\;\RPP\\
x(0)=y_0,
\end{cases}
\end{equation}
which converges weakly to some point $x_\infty\in S$ as
$t\to\pinf$ \cite[Theorem~4]{Bruc75}, and one may therefore 
expect De Pierro's conjecture to hold with $\overline{x}=x_\infty$
under suitable assumptions. Note, however, that for
non-affine convex sets the limit $x_\infty$ might not coincide
with the projection $P_Sy_0$.

\subsection{Under-relaxed cyclic iterations and semigroup flows}
\label{s32_T}

In order to study \eqref{e:2012-04-09a} for a fixed
$\varepsilon\in\zeroun$, it suffices to consider
the iterates modulo $m$, that is, the sequence 
$(y^\varepsilon_{km})_{k\in\NN}=((R^{\varepsilon})^ky_0)_{k\in\NN}$,
which converge weakly towards some point
$\xe_m\in\Fix R^{\varepsilon}$ (see Proposition~\ref{p:2013-04-22}).
The key to establish a formal connection between the iteration
\eqref{e:2012-04-09a} and the semigroup associated with 
\eqref{gradflow_T}, is the following approximation
lemma that relates $R^{\varepsilon}$ to $A=\Id-T$.

\begin{lemma}
\label{lema31_T}    
Set $A=\Id-T$, fix $z\in D$, and set 
$\rho=\max_{i\in I}\|T_iz-z\|/2$. Then
\begin{equation}
\label{ReA}
(\forall\varepsilon\in [0,1])(\forall x\in D)
\quad\|R^{\varepsilon} x-x+\varepsilon m Ax\|\leq
\varepsilon^2(3^m-2m-1)(\|x-z\|+\rho).
\end{equation}
\end{lemma}
\begin{proof}
Since the case $\varepsilon=0$ is trivial, we take 
$\varepsilon\in\rzeroun$. Define operators on $D$ by
\begin{equation}
(\forall j\in I)\quad R^\varepsilon_j=
\big(\Id+\varepsilon(T_j-\Id)\big)\circ\cdots\circ
\big(\Id+\varepsilon(T_{1}-\Id)\big)
\end{equation}
and
\begin{equation}
\label{e:rc12}
(\forall j\in I)\quad E^\varepsilon_j=
\frac{1}{\varepsilon^2}(R^\varepsilon_j-\Id)
+\frac{1}{\varepsilon}\sum_{i=1}^j(\Id-T_i).
\end{equation}
Then $R^{\varepsilon}=R^\varepsilon_m$ and therefore
$R^{\varepsilon}-\Id+\varepsilon m A=\varepsilon^2
E_m^\varepsilon$. Thus, the result boils down to showing that
$(\forall x\in D)$ 
$\|E_m^\varepsilon x\|\leq(3^m-2m-1)(\|x-z\|+\rho)$.
We derive from \eqref{e:rc12} that
\begin{equation}
(\forall j\in\{1,\ldots,m-1\})\quad
E^\varepsilon_{j+1}=E^\varepsilon_j+\frac{1}{\varepsilon}
\Big((\Id-T_{j+1})-(\Id-T_{j+1})\circ R^\varepsilon_j\Big).
\end{equation}
Now let $x\in D$. Since the operators 
$(\Id-T_{j})_{1\leq j\leq m-1}$ are 2-Lipschitz, we have
\begin{align}
(\forall j\in\{1,\ldots,m-1\})\quad
\|E^\varepsilon_{j+1}x\|
&\leq\|E^\varepsilon_j x\|+\frac{2}{\varepsilon}
\|x-R^\varepsilon_jx\|\nonumber\\
&=\|E^\varepsilon_j x\|+2\bigg\|\sum_{i=1}^j(\Id-T_i)x-\varepsilon
E^\varepsilon_j x\bigg\|\nonumber\\
&\leq (1+2\varepsilon)\|E^\varepsilon_j x\|+
2\sum_{i=1}^j(\|x-z\|+\|z-T_iz\|+\|T_iz-T_ix\|)\nonumber\\
&\leq (1+2\varepsilon)\|E^\varepsilon_j x\|+4j(\|x-z\|+\rho).
\label{recur_T}
\end{align}
Using \eqref{recur_T} recursively, and observing that
$E_1^\varepsilon x=0$, it follows that
\begin{equation}
\label{cotita_T}
\|E_m^\varepsilon x\|\leq 4(\|x-z\|+\rho)\sum_{j=1}^{m-1}
j(1+2\varepsilon)^{m-1-j}.
\end{equation}
Upon applying the identity
$\sum_{j=1}^{m-1}j\alpha^j=
((m-1)\alpha^{m+1}-m\alpha^m+\alpha)/(1-\alpha)^2$
to $\alpha=(1+2\varepsilon)^{-1}\in\zeroun$, 
we see that the sum in \eqref{cotita_T} is equal to
$((1+2\varepsilon)^m-1-2m\varepsilon)/(4\varepsilon^2)$, which
increases with $\varepsilon$ attaining its maximum
$(3^m-2m-1)/4$ at $\varepsilon=1$. This combined
with \eqref{cotita_T} yields the announced bound.
\end{proof}

\begin{remark}
For firmly nonexpansive operators, such as projection operators
onto closed convex sets, the operators $(\Id-T_i)_{i\in I}$ are
nonexpansive and the previous proof can be modified to derive a
tighter bound in \eqref{ReA}, namely
\begin{equation}
\label{firm}
(\forall\varepsilon\in [0,1])(\forall x\in D)
\quad\|R^{\varepsilon} x-x+\varepsilon m Ax\|\leq
\varepsilon^2(2^m-m-1)(\|x-z\|+2\rho).
\end{equation}
\end{remark}

We proceed with the announced connection between
\eqref{e:2012-04-09a} and \eqref{gradflow_T}.
This will be used later to establish De Pierro's conjecture in
several alternative settings. 

\begin{proposition} 
\label{sd_T} 
Let $y_0\in D$, let $x$ be the solution of \eqref{gradflow_T},
suppose that Assumptions~\ref{h:1} and \ref{h:2} are satisfied. 
For every $\varepsilon\in\zeroeta$, set
$(z_k^\varepsilon)_{k\in\NN}=((R^{\varepsilon})^ky_0)_{k\in\NN}$ 
and let $\psi^{\varepsilon}$ be the linear
interpolation of $(z^{\varepsilon}_k)_{k\in\NN}$ given by
\begin{equation}
\label{e:poi}
\big(\forall k\in\NN\big)
\big(\forall t\in[km\varepsilon,(k+1)m\varepsilon[\big)\quad
\psi^{\varepsilon}(t)=z^{\varepsilon}_k+\frac{t-km\varepsilon}
{m\varepsilon}(z^{\varepsilon}_{k+1}-z^{\varepsilon}_k).
\end{equation}
Then $(\forall \bar t\in\RPP)$
$\sup_{0\leq t\leq \bar t}\|\psi^{\varepsilon}(t)-x(t)\|\to 0$ 
when $\varepsilon\to 0$.
\end{proposition}
\begin{proof}
Set $A=\Id-T$, let $\varepsilon\in\zeroeta$, and fix $z\in D$.
The function $\psi^{\varepsilon}$ is differentiable except at the
breakpoints $\menge{km\varepsilon}{k\in\NN}$.  
Now set $(\forall k\in\NN)$
$J_k=\left]km\varepsilon,(k+1)m\varepsilon\right[$.
According to Lemma~\ref{lema31_T}, we have
\begin{equation}
(\forall k\in\NN)(\forall t\in J_k)\quad
(\psi^\varepsilon)'(t)=
\frac{1}{m\varepsilon}(z^{\varepsilon}_{k+1}-z^{\varepsilon}_k)
=\frac{1}{m\varepsilon}(R^{\varepsilon}z^{\varepsilon}_k
-z^{\varepsilon}_k)
=-Az^{\varepsilon}_k+\varepsilon h^\varepsilon_k,
\end{equation}
where $\|h^\varepsilon_k\|\leq(3^m-2m-1)
(\|z^{\varepsilon}_k-z\|+\rho)/m$.
Now set 
\begin{equation}
(\forall k\in\NN)(\forall t\in J_k)\quad
h^\varepsilon(t)=
A\psi^{\varepsilon}(t)-Az^{\varepsilon}_k+\varepsilon
h^\varepsilon_k. 
\end{equation}
Then
\begin{equation}
(\forall k\in\NN)(\forall t\in J_k)\quad
(\psi^{\varepsilon})'(t)
=-A\psi^{\varepsilon}(t)+h^\varepsilon(t).
\end{equation}
Moreover, it follows from \eqref{e:H} that there exists a constant 
$\alpha\in\RPP$ independent from $\varepsilon$ such
that $(\forall k\in\NN)$ $\|h^\varepsilon_k\|\leq\alpha$.
Hence, since $A$ is 2-Lipschitz, there exists $\gamma\in\RPP$
such that
\begin{align}
(\forall k\in\NN)(\forall t\in J_k)\quad
\|h^\varepsilon(t)\|
&\leq 2\|\psi^{\varepsilon}(t)-z^{\varepsilon}_k\|+
\varepsilon\|h^\varepsilon_k\|\nonumber\\
&\leq2\|z^{\varepsilon}_{k+1}-z^{\varepsilon}_k\|+
\varepsilon\|h^\varepsilon_k\| \nonumber\\
&=2\varepsilon m \|-Az^{\varepsilon}_k+\varepsilon
h^\varepsilon_k\|+\varepsilon\|h^\varepsilon_k\| \nonumber\\
&\leq\varepsilon \gamma.
\end{align}
Next, consider the function $\theta\colon\RP\to\RP$ defined by 
$\theta(t)=\|x(t)-\psi^{\varepsilon}(t)\|^2$. Then it 
follows from the monotonicity of $A$ that
\begin{align}
(\forall t\in\RP\smallsetminus\menge{km\varepsilon}{k\in\NN})\quad
\theta'(t)&=2\scal{x(t)-\psi^{\varepsilon}(t)}{x'(t)-
(\psi^{\varepsilon})'(t)}\nonumber\\
&=2\scal{x(t)-\psi^{\varepsilon}(t)}
{A\psi^{\varepsilon}(t)-h^\varepsilon(t)-Ax(t)}\nonumber\\
&\leq2\scal{x(t)-\psi^{\varepsilon}(t)}{-h^\varepsilon(t)}
\nonumber\\
&\leq2\|x(t)-\psi^{\varepsilon}(t)\|\,\|h^\varepsilon(t)\|
\nonumber\\
&\leq2\varepsilon \gamma\sqrt{\theta(t)}.
\end{align}
Integrating this inequality and noting that $\theta(0)=0$, we 
obtain $(\forall t\in\RP)$ 
$\|\psi^{\varepsilon}(t)-x(t)\|=
\sqrt{\theta(t)}\leq\varepsilon\gamma t$. Now let $\bar t\in\RPP$.
Then 
$\sup_{0\leq t\leq \bar t}\|\psi^{\varepsilon}(t)-x(t)\|\leq
\varepsilon\gamma\bar t\to 0$ as $\varepsilon\to 0$.
\end{proof}

\subsection{Strong convergence under stability of 
approximate cycles} 

In this section, we investigate the strong convergence
of the cycles defined in \eqref{e:2011-12-06} when a 
stability condition holds.

\begin{theorem}
\label{p33_T}
Suppose that Assumptions~\ref{h:1} and \ref{h:2} are satisfied, 
and that
\begin{equation}
\label{e:stab}
(\forall z\in\Fix T)\quad
\lim_{\varepsilon\to 0}d_{\Fix R^{\varepsilon}}(z)=0.
\end{equation}
In addition, let $y_0\in D$, and suppose that 
the orbit of $y_0$ in the Cauchy problem 
\eqref{gradflow_T} converges strongly, say 
$x(t)\to\overline{x}\in D$ 
as $t\to\pinf$. For every $\varepsilon\in\zeroeta$, let
$(x_i^\varepsilon)_{i\in I}$
be the cycle obtained as the weak limit of \eqref{e:2012-04-09a} 
in Proposition~\ref{p:2013-04-22}. 
Then $\overline{x}\in\Fix T$ and $(\forall i\in I)$ 
$x_i^\varepsilon\to\overline{x}$ when $\varepsilon\to 0$.
\end{theorem}
\begin{proof}
Since $x(t)\to\overline{x}$, \eqref{gradflow_T} implies that $x'(t)$
converges to $A\overline{x}$ and therefore $A\overline{x}=0$ since
$x'(t)\to 0$. Hence, $\overline{x}\in\Fix T$. Now fix 
$\delta\in\RPP$ and $\bar t\in\RPP$ such that 
$(\forall t\in[\bar t,\pinf[)$ 
$\|x(t)-\overline{x}\|\leq\delta$. 
For every $\varepsilon\in\zeroeta$,
set $(z_k^\varepsilon)_{k\in\NN}=(y^{\varepsilon}_{km})_{k\in\NN}
=((R^{\varepsilon})^ky_0)_{k\in\NN}$ 
and define the function $\psi^{\varepsilon}$ as in \eqref{e:poi}.
By Proposition~\ref{sd_T}, there exists $\varepsilon_0\in\zeroeta$ 
such that
\begin{equation}
(\forall\varepsilon\in\left]0,\varepsilon_0\right[)
(\forall t\in\left[0,\bar t+m\right])\quad
\|\psi^{\varepsilon}(t)-x(t)\|\leq\delta. 
\end{equation}
Now let $\varepsilon\in\left]0,\varepsilon_0\right[$, choose 
$k_0\in\NN$ such that $k_0 m\varepsilon\in[\bar t,\bar t+m]$,
and set $\bar{x}^\varepsilon=P_{\Fix R^{\varepsilon}}\overline{x}$
(recall that, since $D$ is closed and convex and $R^{\varepsilon}$ 
is nonexpansive, $\Fix R^{\varepsilon}$ is closed and convex
\cite[Corollary~4.15]{Livre1}).
Then $\|z^{\varepsilon}_{k_0}-x(k_0 m\varepsilon)\|=
\|\psi^{\varepsilon}(k_0 m \varepsilon)-x(k_0m\varepsilon)\|
\leq\delta$
and therefore $\|z^{\varepsilon}_{k_0}-\overline{x}\|\leq 2\delta$. 
Since $R^{\varepsilon}$ is nonexpansive, we have
$(\forall k\in\NN)$ $\|z^{\varepsilon}_{k+1}-\bar{x}^\varepsilon\|
\leq\|z^{\varepsilon}_k-\bar{x}^\varepsilon\|$. Hence, for every
integer $k\geq k_0$, we have
\begin{equation}
\|z^{\varepsilon}_k-\bar{x}^\varepsilon\|\leq
\|z^{\varepsilon}_{k_0}-\bar{x}^\varepsilon\|\leq
\|z^{\varepsilon}_{k_0}-\overline{x}\|+
\|\overline{x}-\bar{x}^\varepsilon\|\leq 2\delta
+d_{\Fix R^\varepsilon}(\overline{x})
\end{equation}
and therefore
\begin{equation}
\|y^{\varepsilon}_{km}-\overline{x}\|=
\|z^{\varepsilon}_k-\overline{x}\|\leq 2\delta
+2d_{\Fix R^\varepsilon}(\overline{x}).
\end{equation}
Since Proposition~\ref{p:2013-04-22} asserts that 
$y^{\varepsilon}_{km}\weakly\xe_m$, we get
\begin{equation}
\label{e:2013-04-22h}
\|\xe_m-\overline{x}\|\leq\varliminf_{k\to\pinf}
\|y^{\varepsilon}_{km}-\overline{x}\|
\leq 2\delta+2d_{\Fix R^\varepsilon}(\overline{x}),
\end{equation}
and \eqref{e:stab} yields
\begin{equation}
\varlimsup_{\varepsilon\to 0}
\|x^\varepsilon_m-\overline{x}\|\leq 2\delta.
\end{equation}
Letting $\delta\to 0$, we deduce that $\xe_m\to\overline{x}$ as
$\varepsilon\to 0$. In turn, 
it follows from \eqref{e:2013-03-24} that $(\forall i\in I)$
$\xe_i\to\overline{x}$ as $\varepsilon\to 0$.
\end{proof}

The following corollary settles entirely De Pierro's conjecture in
the case of $m=2$ closed convex sets in Euclidean spaces.

\begin{corollary} 
In Assumption~\ref{h:1}, suppose that $\HH$ is finite-dimensional,
$D=\HH$, and $m=2$, and let $T_1=P_1$ and $T_2=P_2$ be the 
projection operators onto nonempty closed convex sets such that
\begin{equation}
\label{e:S1}
\Fix T=S=\operatorname{Argmin}\Phi\neq\emp,
\quad\text{where}\quad\Phi=\frac{1}{4}\big(d_{C_1}^2+d_{C_2}^2\big).
\end{equation}
Let $y_0\in\HH$ and let $\overline{x}\in S$ be the limit of the the
solution $x$ of Cauchy problem
\begin{equation}
\begin{cases}
x'(t)+x(t)=\frac{1}{2}\big(P_1x(t)+P_2x(t)\big)
\;\,\text{on}\;\RPP\\
x(0)=y_0.
\end{cases}
\end{equation}
For for every $\varepsilon\in\zeroun$, let
$x_1^\varepsilon=\lim_{k\to\pinf}y^\varepsilon_{2k+1}$ and 
$x_2^\varepsilon=\lim_{k\to\pinf}y^\varepsilon_{2k+2}$, where
\begin{equation}
\label{e:2013-04-23a}
(\forall k\in\NN)\quad 
\begin{array}{l}
\left\lfloor
\begin{array}{ll}
y^\varepsilon_{2k+1}&=\big(\Id+\varepsilon(P_{1}-\Id)\big)
y^\varepsilon_{2k}\\[2mm]
y^\varepsilon_{2k+2}&=\big(\Id+\varepsilon(P_{2}-\Id)\big)
y^\varepsilon_{2k+1}.
\end{array}
\right.
\end{array}
\end{equation}
Then $x_1^\varepsilon\to\overline{x}$ and 
$x_2^\varepsilon\to\overline{x}$ when $\varepsilon\to 0$.
\end{corollary}
\begin{proof} 
Fix $z\in S$, and set $a=P_1z$ and $b=P_2z$. Then $z=(a+b)/2$
and $(\forall\varepsilon\in\zeroun)$
$z^{\varepsilon}=((1-\varepsilon) a+b)/(2-\varepsilon)\in
\Fix R^{\varepsilon}$. Thus 
\begin{equation}
\label{m=2}
d_{\Fix R^{\varepsilon}}(z)\leq\|z-z^{\varepsilon}\|=
\frac{\varepsilon\|b-a\|}{2(2-\varepsilon)}\to 0
\quad\text{as}\quad\varepsilon\to 0,
\end{equation}
and the conclusion follows from Theorem~\ref{p33_T}.
\end{proof}

\begin{figure}[t]
\centering
\includegraphics[width=7cm]{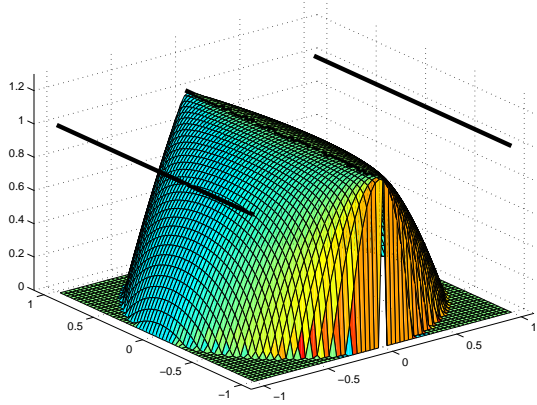}
\caption{An example in which the condition \eqref{e:stab} fails.}
\label{fig:2}
\end{figure}

We conclude this section by showing that, in contrast with 
\eqref{m=2}, the condition \eqref{e:stab} can fail in the case 
of projection operators in the presence of $m=3$ sets.

\begin{example}
\label{contraejemplo}
Suppose that $\HH=\RR^3$ and $m=3$, and let $T_1$, $T_2$, and
$T_3$ be, respectively, the projection operators onto the bounded
closed convex sets (see Fig.~\ref{fig:2})
\begin{equation}
\begin{cases}
C_1=[-1,1]\times\{-1\}\times\{1\}\\
C_2=[-1,1]\times\{1\}\times\{1\}\\
C_3=\menge{(\xi_1,\xi_2,\xi_3)\in\RR^3}
{\xi_1\in[-1,1],\,\xi_3\in[0,1],\,
(1-\xi_3)(\xi_1^2-1)+\xi_2^2\leq 0}.
\end{cases}
\end{equation}
Then the set of least-squares solutions is 
$S=\Fix T=[-1,1]\times\{0\}\times\{1\}\subset C_3$. Moreover, 
\begin{equation}
(\forall\varepsilon\in\zeroun)\quad
\Fix R^\varepsilon=\{z^\varepsilon\}=
\left\{\left(0,\frac{w_\varepsilon+\varepsilon(1-\varepsilon)}
{3(1-\varepsilon)+\varepsilon^2},1-\frac{w_\varepsilon^2}
{3(1-\varepsilon)+\varepsilon^2}\right)\right\},
\end{equation}
where $w_\varepsilon$ is the unique real solution of
$2w^3+w=\varepsilon/(2-\varepsilon)$.
Clearly $z^\varepsilon\to(0,0,1)\in S$ as
$\varepsilon\to 0$, but 
$(\forall z\in S\smallsetminus\{(0,0,1)\})$ 
$d_{\Fix R^\varepsilon}(z)\not\to 0$ as $\varepsilon\to 0$.
\end{example}

\subsection{Strong convergence under local strong monotonicity}

Another situation covered by Theorem~\ref{p33_T} is when the
operator $T$ has a unique fixed point $\overline{x}$ and $A=\Id-T$ 
is locally strongly monotone around $\overline{x}$,  namely
\begin{equation}
\label{M_T}
(\exi\alpha\in\RPP)(\exi\delta\in\RPP)
(\forall x\in D\cap B(\overline{x};\delta))\quad
\scal{x-\overline{x}}{x-Tx}\geq\alpha\|x-\overline{x}\|^2.
\end{equation}
In the case of convex projections operators, then 
$A=\nabla\Phi$ and, if $\Phi$
is twice differentiable at $\overline{x}$, then \eqref{M_T} is 
equivalent to the positive-definiteness of 
$\nabla^2\Phi(\overline{x})$.
Another case in which \eqref{M_T} is satisfied, with 
$\alpha=1-\rho$, is when $T$
is a local strict contraction  with constant $\rho\in\zeroun$
at the fixed point $\overline{x}$, namely, for all
$x$ in some ball $B(\overline{x};\delta)$,
$\|Tx-T\overline{x}\|\leq\rho\|x-\overline{x}\|$.
If $T$ is differentiable at $\overline{x}$ this
amounts to $\|T'(\overline{x})\|<1$.

\begin{theorem} 
\label{p38_T}
Suppose that Assumptions~\ref{h:1} and \ref{h:2} are satisfied,
together with \eqref{M_T}, and let $\Fix T=\{\bar{x}\}$. In 
addition, let $y_0\in D$ and, for every $\varepsilon\in\zeroeta$, 
let $(x_i^\varepsilon)_{i\in I}$ be the cycle obtained as the weak 
limit of \eqref{e:2012-04-09a} in Proposition~\ref{p:2013-04-22}. 
Then $(\forall i\in I)$ $x_i^\varepsilon\to\overline{x}$ as 
$\varepsilon\to 0$.
\end{theorem}
\begin{proof}
It suffices to check the assumptions of Theorem~\ref{p33_T}. 
Set $A=\Id-T$ and let $x$ be the solution to \eqref{gradflow_T}.
\begin{itemize}
\item
$d_{\Fix R^{\varepsilon}}(\overline{x})\to 0$ as 
$\varepsilon\to 0$:\\
Let $\varepsilon\in\left]0,\min\{\eta,\alpha/(2m)\}\right[$,
set $Q^\varepsilon=\Id-m\varepsilon A$ and 
$\gamma(\varepsilon)=1-m\varepsilon(\alpha-2m\varepsilon)$,
and let $y\in D\cap B(\overline{x};\delta)$.
Since $A\overline{x}=0$ and $A$ is 2-Lipschitz, we have
\begin{align}
\label{e:KJ2013a}
\|Q^\varepsilon y-\overline{x}\|^2
&=\|y-\overline{x}\|^2-2m\varepsilon\scal{y-\overline{x}}
{Ay-A\overline{x}}+(m\varepsilon)^2\|Ay-A\overline{x}\|^2\nonumber\\
&\leq(1-2m\varepsilon(\alpha-2m\varepsilon))\|y-\overline{x}\|^2 
\nonumber\\
&\leq\gamma(\varepsilon)^2\|y-\overline{x}\|^2.
\end{align}
On the other hand, setting 
$\rho=\max_{i\in I}\|T_i\overline{x}-\overline{x}\|/2$ 
and $\beta=3^m-2m-1$, Lemma~\ref{lema31_T} gives 
\begin{equation}
\label{e:KJ2013b}
\|R^{\varepsilon}
y-Q^\varepsilon y\|\leq \varepsilon^2 \beta(\|y-\overline{x}\|+\rho)
\end{equation}
which, combined with \eqref{e:KJ2013a}, yields
\begin{equation}
\label{e:KJ2013c}
\|R^{\varepsilon}y-\overline{x}\|\leq
\|R^{\varepsilon}y-Q^\varepsilon y\|+\|Q^\varepsilon y-\overline{x}\|
\leq\varepsilon^2\beta(\|y-\overline{x}\|+\rho)
+\gamma(\varepsilon)\|y-\overline{x}\|.
\end{equation}
From this estimate it follows that given 
$\delta'\in\left]0,\delta\right]$, for every
$\varepsilon\leq m\alpha\delta'/(\beta(\delta'+\rho)+2m^2\delta')$,
we have 
$R^{\varepsilon}(D\cap B(\overline{x};\delta'))\subset D\cap
B(\overline{x};\delta')$. Therefore $R^\varepsilon$ has a fixed 
point in $B(\overline{x};\delta')$ and hence
$d_{\Fix R^{\varepsilon}}(\overline{x})\leq\delta'$.
Since $\delta'$ can be arbitrarily small, this proves
that $d_{\Fix R^{\varepsilon}}(\overline{x})\to 0$ 
as $\varepsilon\to 0$.
\item
$x(t)\to\overline{x}$ as $t\to\pinf$:\\
Let $\theta\colon\RP\to\RP$ be defined by 
$\theta(t)=\|x(t)-\overline{x}\|^2/2$,
and let us show that $\lim_{t\to\pinf}\theta(t)=0$.
We note that this holds whenever the orbit enters the ball 
$B(\overline{x};\delta)$ at some instant $t_0$. Indeed, 
the monotonicity of $A$ implies that $\theta$ is decreasing 
so that, for every $t\in[t_0,\pinf[$, $x(t)\in D\cap
B(\overline{x};\delta)$ and hence \eqref{gradflow_T} and 
\eqref{M_T} give
\begin{equation}
\theta'(t)=\scal{x(t)-\overline{x}}{x'(t)}
=\scal{\overline{x}-x(t)}{x(t)-Tx(t)}\leq-\alpha\|x(t)-
\overline{x}\|^2=-2\alpha\theta(t).
\end{equation}
Consequently, $\theta(t)\leq\theta(t_0)\exp(-2\alpha(t-t_0))\to 0$ as
$t\to\pinf$. It remains to prove that $x(t)$ enters the ball 
$B(\overline{x};\delta)$. If this was 
not the case we would have 
$\mu=\lim_{t\to\pinf}\sqrt{\theta(t)}\geq\delta$. Choose $t_0$ large 
enough so that $\sqrt{\theta(t_0)}\leq\mu+{\delta}/{2}$ and let
$\tilde x$ be the solution to the Cauchy problem
\begin{equation}
\begin{cases}
\tilde{x}'(t)=-A\tilde{x}(t)\;\,\text{on}\;[t_0,\pinf[\\
\tilde{x}(t_0)=\tilde x_0,
\end{cases}
\end{equation}
where 
$\tilde x_0=\overline{x}+\delta(x(t_0)-\overline{x})/
\|x(t_0)-\overline{x}\|\in D\cap B(\overline{x};\delta)$.
By monotonicity of $A$, $t\mapsto\|x(t)-\tilde x(t)\|$ is
decreasing and hence
\begin{align}
(\forall t\in[t_0,\pinf[)\quad
\|x(t)-\overline{x}\|
&\leq\|x(t)-\tilde x(t)\|+\|\tilde x(t)-\overline{x}\| \nonumber\\
&\leq\|x(t_0)-\tilde x(t_0)\|+\|\tilde x(t)-\overline{x}\| \nonumber\\
&\leq(\mu-\delta/2)+\|\tilde x(t)-\overline{x}\|.
\end{align}
Since by the previous argument $\|\tilde x(t)-\overline{x}\|\to 0$, 
we reach a contradiction with the fact that 
$(\forall t\in\RP)$ $\|x(t)-\overline{x}\|\geq\mu$.
\end{itemize}
Altogether, the conclusion follows from Theorem~\ref{p33_T}.
\end{proof}

\begin{remark}
If $\Id-T$ were globally (rather than just locally
as in \eqref{M_T}) strongly monotone 
at every point in $\Fix T$, we could derive 
Theorem~\ref{p38_T} directly from Theorem~\ref{t:1} and
Remark~\ref{r:2013-04-23}\ref{r:2013-04-23ii}.
\end{remark}

Theorem~\ref{p38_T} can also be applied when the local strong
monotonicity or the local contraction properties hold up to an
affine subspace (see \eqref{paralelo_T} below). 
This is relevant in the case studied in
\cite{Baus05} when  $(T_i)_{i\in I}$ is a family of projection 
operators onto closed affine subspaces $(x_i+E_i)_{i\in I}$,
where $(E_i)_{i\in I}$ is a family of closed vector subspaces of
$\HH$, and more generally for unbounded closed convex cylinders of 
the form $(B_i+E_i)_{i\in I}$, where $B_i$ is a nonempty bounded
closed convex subset of $E_i^\perp$.

\begin{corollary}
\label{corf}
Suppose that Assumptions~\ref{h:1} and \ref{h:2} are satisfied,
that $D=\HH$, and that $(T_i)_{i\in I}$ is a family of projection 
operators onto nonempty closed convex subsets $(C_i)_{i\in I}$ of 
$\HH$. In addition, suppose that the set $S$ of minimizers of 
$\Phi$ in \eqref{e:1994} is a closed affine subspace, say 
$S=z+E$, where $z\in\HH$ and $E$ is a closed vector subspace of 
$\HH$. Let $y_0\in D$, set $\overline{x}=P_Sy_0$, and, for every 
$\varepsilon\in\zeroeta$, let $(x_i^\varepsilon)_{i\in I}$ be the 
cycle obtained as the weak limit of \eqref{e:2012-04-09a} in 
Proposition~\ref{p:2013-04-22}. Then the following hold.
\begin{enumerate}
\item
\label{corfi}
$(\forall i\in I)$ $x_i^\varepsilon\weakly\overline{x}$ as
$\varepsilon\to 0$.
\item
\label{corfii}
Suppose that
\begin{equation}
\label{paralelo_T}
(\forall y\in S)(\exi\rho\in [0,1[)(\exi\delta\in\RPP)
(\forall x\in B(0;\delta)\cap E^\perp)\;\;
\|T(x+y)-Ty\|\leq\rho\|x\|.
\end{equation}
Then $(\forall i\in I)$ $x_i^\varepsilon\to\overline{x}$ as
$\varepsilon\to 0$.
\end{enumerate}
\end{corollary}
\begin{proof}
Let $i\in I$. Since $S=z+E$, we have $C_i+E\subset C_i$ and
the iterates $(y_{k}^\varepsilon)_{k\in\NN}$ in 
\eqref{e:2012-04-09a} move parallel to $E^\perp$ and remain in
$y_0+E^\perp$.
Hence, since $\{\overline{x}\}=S\cap (y_0+E^\perp)$, 
\ref{corfi} follows by applying Theorem~\ref{t:1} in the space 
$y_0+E^\perp$, while \ref{corfii} follows by applying 
Theorem~\ref{p38_T} in this same space.
\end{proof}

We conclude the paper by revisiting De Pierro's conjecture in the
affine setting investigated in \cite{Baus05}. More precisely, we
shall derive an alternative proof of the main result of
\cite{Baus05} from Corollary~\ref{corf}. For this purpose, we need
the following notion of regularity.

\begin{definition}
A finite family $(E_i)_{i\in I}$ of closed vector subspaces of $\HH$ 
with intersection $E$ is regular if 
\begin{equation}
\big(\forall (y_k)_{n\in\NN}\in\HH^{\NN}\big)\quad
\max_{i\in I}d_{E_i}(y_k)\to 0\quad\Rightarrow\quad d_E(y_k)\to 0.
\end{equation}
\end{definition}

\begin{theorem} 
Let $(E_i)_{i\in I}$ be a regular family of closed vector subspaces of 
$\HH$ with intersection $E$ and for, every $i\in I$, let 
$\overline{x}_i\in\HH$ and let $P_i$ be the projection operator onto 
the affine subspace $C_i=\overline{x}_i+E_i$. Let $y_0\in\HH$ and set 
$S=\operatorname{Argmin}\sum_{i\in I}d_{C_i}^2$. Then there exists
$z\in\HH$ such that $S=z+E$. Moreover, for every 
$\varepsilon\in\rzeroun$, the cycle $(x_i^\varepsilon)_{i\in I}$
obtained as the weak limit of \eqref{e:2012-04-09a} in 
Proposition~\ref{p:2013-04-22} exists, and $(\forall i\in I)$
$x_i^\varepsilon\to P_Sy_0$ as $\varepsilon\to 0$.
\end{theorem}
\begin{proof} 
We have $(\forall i\in I)$ 
$P_i\colon x\mapsto\overline{x}_i+P_{E_i}(x-\overline{x}_i)$.
Hence $Tx=a+Lx$, where $a=(1/m)\sum_{i\in I}(\overline{x}_i-
P_{E_i}\overline{x}_i)$ and $L=(1/m)\sum_{i\in I}P_{E_i}$. 
According to \cite[Theorem~5.4]{Baus05}, the subspaces 
$(E_i)_{i\in I}$ are regular if and only if 
$\rho=\|L\circ P_{E^\perp}\|<1$, which implies that $T$ is a
strict contraction on $y_0+E^\perp$. From this we
get simultaneously that $T$ has a fixed point $z$, that the
least-squares solution set is of the form $S=z+E$, and that
\eqref{paralelo_T} holds. Hence, the result will follow from
Corollary~\ref{corf} provided that $(x_i^\varepsilon)_{i\in I}$ 
exists for every $\varepsilon\in\rzeroun$.
This was proved in \cite[Theorem~5.6]{Baus05} by noting that
$R^\varepsilon|_{y_0+E^\perp}$ is a strict contraction. Indeed,
$R^\varepsilon$ is a composition of affine maps and an inductive
calculation reveals that it can be written as $R^\varepsilon
x=a^\varepsilon+L^\varepsilon x$, where $a^\varepsilon\in\HH$ and
$L^\varepsilon$ a linear operator which is a convex combination of
nonexpansive linear maps, one of which is the strict contraction 
$L\circ P_{E^\perp}$.
\end{proof}

\begin{remark}
Corollary~\ref{corf}\ref{corfi} seems to be new even for affine 
subspaces $(C_i)_{i\in I}$. Also new in 
Corollary~\ref{corf}\ref{corfii} 
is the fact that strong convergence holds for more general convex 
sets than just translates of regular subspaces.
\end{remark}

\vspace{2ex}
\noindent
{\bf Acknowledgement.}
The research of P. L. Combettes was supported in part by
the European Union under the 7th Framework Programme
``FP7-PEOPLE-2010-ITN'', grant agreement number
264735-SADCO.
The research of R. Cominetti was supported by Fondecyt 1130564 and 
N\'ucleo Milenio Informaci\'on y Coordinaci\'on en Redes 
ICM/FIC P10-024F.

\end{document}